\documentclass[12pt]{article}
\usepackage{amsmath,amsthm,amsfonts,amssymb}
\usepackage{hyperref}

\newtheorem{theorem}{Theorem}
\newtheorem{lemma}[theorem]{Lemma}

\theoremstyle{remark}
\newtheorem{remark}[theorem]{Remark}
\theoremstyle{definition}

\newcommand{\FF}{\mathbb{F}}
\newcommand{\PG}{\text{PG}}

\title{A construction for clique-free  pseudorandom graphs}
\author{Anurag Bishnoi\thanks{Department of Mathematics and Statistics, The University of Western Australia, Perth, Australia, \href{mailto:anurag.2357@gmail.com}{anurag.2357@gmail.com}. Research supported in part by a Humboldt Research Fellowship for Postdoctoral Researchers and by Discovery Early Career Award of the Australian Research Council (No.~DE190100666).}, Ferdinand Ihringer\thanks{Department of Mathematics: Analysis, Logic and Discrete Mathematics, Ghent University, Belgium, \href{mailto:ferdinand.ihringer@ugent.be}{ferdinand.ihringer@ugent.be} . The author is supported by a postdoctoral fellowship of the Research Foundation --- Flanders (FWO).}, Valentina Pepe\thanks{Department of Basic and Applied Sciences for Engineering, Sapienza University of Rome. The author is supported by INDAM (Istituto Nazionale Di Alta Matemetica)}
}

\begin{document}

\maketitle

\begin{abstract}
A construction of Alon and Krivelevich gives highly pseudorandom $K_k$-free graphs on $n$ vertices with edge density equal to $\Theta(n^{-1/(k -2)})$.
In this short note we improve their result by constructing an infinite family of highly pseudorandom $K_k$-free graphs with a higher edge density of  $\Theta(n^{-1/(k - 1)})$.
\end{abstract}

\section{Introduction}
Pseudorandom graphs are deterministic graphs that in some sense behave like random graphs.
They have played an important role in modern graph theory and theoretical computer science.
We refer to the survey by Krivelevich and Sudakov \cite{KrivSu06} for background and applications of pseudorandom graphs.

One well studied measure of pseudorandomness is in terms of the eigenvalues (of the adjacency matrix) of a graph.
A graph is called an \textit{$(n, d, \lambda)$-graph} if it is a $d$-regular graph on $n$ vertices, with its second largest eigenvalue
(in absolute value) at most $\lambda$.
By looking at the trace of the square of the adjacency matrix, we see that $\lambda = \Omega(\sqrt{d})$ for any such graph, whenever, say, $d < n/2$.
Graphs which have $\lambda = O(\sqrt{d})$ are known as \textit{optimally pseudorandom graphs}.

In \cite{Alon94}, Alon constructed a family of dense triangle-free optimally pseudorandom graphs to provide explicit graphs for a lower bound on the off-diagonal Ramsey number $R(3, t)$ (see \cite{Alon_survey} for a recent survey on many applications of his construction, and \cite{Conlon, Kopparty} for alternate constructions).
Alon's family is in fact extremal in the sense that any triangle-free $(n, d, \lambda)$-graph with $\lambda = O(\sqrt{d})$ must have $d/n = O(n^{-1/3})$, and Alon's family satisfies $d/n = \Omega(n^{-1/3})$.
The natural extension of this is to look at $K_k$-free graphs with the largest possible \textit{edge-density} $d/n$.
A simple application of the expander mixing lemma (cf.~\cite[Thm.~2.11]{KrivSu06}) shows that any $K_k$-free $(n, d, \lambda)$-graph with $\lambda = O(\sqrt{d})$ satisfies
\[\frac{d}{n} = O\left(n^{ \frac{-1}{2k - 3}}\right).\]
Several people have asked about the tightness of this bound \cite{Alon_survey, Conlon-Lee19, Fox-Lee-Sudakov13, Kri-Su-Sz04, Sudakov-Szabo-Vu05}.
Despite years of effort, this problem remains open for every $k \geq 4$.
The best known general construction for such graphs is the $20$ year old construction by Alon and Krivelevich \cite{Alon1997}, where the density $d/n$ is equal to $\Theta(n^{-1/(k - 2)})$.
As a step towards the conjecture, Conlon and Lee suggest that ``A first aim would be to beat the construction of Alon and Krivelevich'' \cite[Sec~6]{Conlon-Lee19}.
We provide such an improvement by constructing an infinite family of $K_k$-free optimally pseudorandom graphs with $d/n = \Theta(n^{-1/(k - 1)})$.

\section{Construction}

Our construction makes use of the finite geometry associated with a quadratic form over a finite field, see \cite{Ball2015,Munemasa1996}
for the general theory behind it. We repeat the relevant facts in the following.
We denote the $(k-1)$-dimensional projective space over $\FF_q$ by $\PG(k-1, q)$.
The \textit{points} of $\PG(k-1, q)$ are the $1$-spaces of $\FF_q^k$. 
As there are $q^k-1$ non-zero vectors in $\FF_q^k$, each non-zero vector spans a $1$-space, and
each $1$-space contains $q-1$ non-zero vectors, $\PG(k-1, q)$ contains $(q^k-1)/(q-1) = (1 + o(1))q^{k - 1}$ points.
We assume in the following that $q$ is the power of an \textit{odd} prime.

Define the quadratic form $Q: \FF_q^k \rightarrow \FF_q$ with $Q(x_1, \dots, x_k) = \xi x_1^2 +  \sum_{i=2}^k x_i^2$, where $\xi$ is a non-square of $\FF_q$, which exists because $q$ is odd.
We say that two points $x = \langle (x_1, \ldots, x_k ) \rangle$ and $y = \langle (y_1, \ldots, y_k) \rangle$
of $\PG(k-1, q)$ are orthogonal, denoted by $x \in y^\perp$, if \[\frac{1}{2}(Q(x+y) - Q(x) - Q(y)) = \xi  x_1y_1 +  \sum_{i=2}^k x_iy_i = 0.\]
Let $X_0$ be the set of \textit{singular points}, $X_\square$ be the set of square points and $X_\boxtimes$ be the set of non-square points of $\PG(k-1, q)$ with respect to $Q$, that is
\begin{align*}
  &X_0 = \{ x   \in \PG(k-1, q): Q(x) = 0 \},\\
  &X_\square = \{ x   \in \PG(k-1, q): Q(x) \text{ is a square in } \FF_q \setminus \{ 0 \} \},\\
  &X_\boxtimes = \{ x  \in \PG(k-1, q): Q(x) \text{ is a non-square in } \FF_q \}.
\end{align*}
Note that these point-sets of $\PG(k - 1, q)$ are well defined since $Q(\lambda x) = \lambda^2 Q(x)$, and hence being a square or not is a property of $1$-spaces of $\mathbb{F}_q^k$.
Let $\Gamma^\epsilon(k, q)$ be the graph with vertex set $X_\epsilon$ where two vertices $x$ and $y$ are adjacent if $x \in y^\perp$, for $\epsilon \in \{\square, \boxtimes\}$.
The graph $\Gamma^\square(k, q)$ will be our $K_k$-free pseudorandom graph.

\begin{remark}
These objects naturally belong to finite classical groups and they have been studied in the literature for more than 80 years (cf.~\cite{Witt1937}).
One can even argue that Jordan already understood them in 1870 \cite{Jordan1870}.
For a study of the associated graphs see for example Bannai et al.~\cite{Bannai90}.
For $k$ odd, $\Gamma^\epsilon(k, q)$ is
either the graph with vertex set $\Omega_1$ and adjacency relation $R_{(q+1)/2}$ in \cite[Sec.~6]{Bannai90}
or the graph with vertex set $\Omega_2$ and adjacency relation $R_{(q+1)/2}$ in \cite[Sec.~7]{Bannai90}.
For $k$ even, $\Gamma^\epsilon(k, q)$ either corresponds to the graph with adjacency relation $R_{(q+1)/2}$ in \cite[Sec.~4]{Bannai90} or \cite[Sec.~5]{Bannai90}.
Note that \cite{Bannai90} uses a different quadratic form (see Remark \ref{rem:choice_Q}), 
so $\Gamma^\square(k, q)$ can be isomorphic to their graph on non-zero squares or non-squares.
According to \cite{Bannai90}, the results of \cite{Bannai90}, which we use, can also be obtained from Soto-Andrade's work
in \cite{SotoAndrade1971,SotoAndrade1985} for $k$ even. 
Our graphs are also mentioned by Hubaut \cite[\S8.10~and~p.~377]{Hubaut1975} for $q=3$, and for $q=5$ when $n$ is odd by Willbrink \cite[\S7.D]{Brouwer1984} as they are strongly regular graphs.
\end{remark}

\begin{remark}
Our graphs are very similar to the ones used by by Alon and Krivelevich \cite[Sec.~2]{Alon1997}.
The differences are that (1) we use the bilinear form $\xi x_1y_1 +  \sum_{i=2}^k x_iy_i$ to define adjacency
while they use $\sum_{i=1}^k x_iy_i$; and that (2) our vertices are the points in $X_\square$,
while their vertices are all points $x$ of $\PG(k-1, q)$ with $x \notin x^\perp$, that is, the points in $X_\square \cup X_\boxtimes$.
We will see that $\Gamma^\square(k, q)$ has almost the same edge density as the Alon-Krivelevich graph, but while their graph is only $K_{k + 1}$-free (and has plenty of $K_k$'s) our graph is $K_k$-free.
\end{remark}

\begin{remark}\label{rem:choice_Q}
  It follows from the general theory of quadratic forms (cf.~\cite{Ball2015,Munemasa1996,Witt1937}) that our choice of $Q$
  does not matter (much) as there are only two isometry types of non-degenerate quadratic forms on $\FF_q^k$.
  The isometry type depends on the discriminant of the form.
  For $k$ odd, for any non-degenerate quadratic form either the graph on non-zero squares or the graph on non-squares is $K_k$-free.
  For $k$ even, for one type of non-degenerate quadratic form the graph on non-zero squares and the graph on non-squares
  are both $K_k$-free. In each case, the corresponding $K_k$-free graph is isomorphic to our graph with the same parameters.
\end{remark}

It is well-known, see for instance Proposition 4.1 in \cite{Munemasa1996}, that $|X_0| = (1+o(1)) q^{k-2}$.
Hence, $|X_\square \cup X_\boxtimes| = (q^k - 1)/(q - 1) - |X_0| = (1+o(1)) q^{k-1}$.
The number of vertices $|X_\epsilon|$ is in fact $(1 + o(1)) q^{k-1}/2$.
The number of solutions of $Q(x) = a$ is roughly the same for any $a$. As $(q-1)/2$ of the elements of $\FF_q$
are non-zero squares, and $(q-1)/2$ of the elements of $\FF_q$ are non-square, $|X_\epsilon| = (1 + o(1)) q^{k-1}/2$ is plausible.
The exact number is exactly the sum of the first row in the corresponding character table in \cite{Bannai90}.
These are the tables VI--IX, Theorem 6.3 and Theorem 7.3.

We use the following fact which can be deduced from the general theory of quadratic forms due to Witt from 1937 \cite{Witt1937}:
\begin{lemma}\label{lem:vtx_trans}
  The graph $\Gamma^\epsilon(k, q)$ is vertex-transitive.
\end{lemma}

Let $V$ be the $\mathbb{F}_q$-vector space $\mathbb{F}_q^k$. The orthogonal group associated to $Q$ is the subgroup of $\mathrm{GL}(V)$ given by $G=\{f \in \mathrm{GL}(V)~|~Q(f(x))=Q(x) \text{ for all }x \in V\}$. 
The group  $G$ obviously preserves the orthogonality relation $\perp$. 
We provide a sketch of a proof that shows that $G$ acts transitively on $X_\square$.
A similar proof works for transitivity on $X_\boxtimes$. 

\begin{proof}[Sketch of proof]
Let $B$ be the matrix associated to $Q$, that is $Q(x)=x^TBx$, and hence we have the bilinear form $\beta(x,y)=\frac{1}{2}(Q(x+y) - Q(x) - Q(y)) = x^TBy$. 
If $A$ is the matrix of $f \in \mathrm{GL}(V)$, then $f \in G$ if and only if $A^TBA=B$. 
Hence, if $c_1,c_2,\ldots,c_k$ are the columns of $A$, then we have $\beta(c_i,c_j)=0$ for all $i \neq j$, $\beta(c_1,c_1)=\xi$ and $\beta(c_i,c_i)=1$ for all $i >1$. 
Let $x=\langle (0,\ldots, 0, 1) \rangle$ and $\langle y \rangle$ any other element of $X_\square$. 
For any subspace $U$ of $\mathrm{PG}(V)$ which is not entirely contained in $X_0$ (that is, a non-singular subspace), $Q$ induces a non-degenerate quadratic form on $U$ and hence $|U \cap X_\epsilon| = (1 + o(1))q^{\dim U}/2$ for $\epsilon \in \{\square, \boxtimes\}$. 
Let $c_k = y$ and iteratively pick $c_{k - 1}, c_{k - 2}, \dots, c_1$ such that $c_i \in \langle c_{i + 1}, \dots, c_k \rangle^\perp$ for all $1 \leq i \leq k - 1$, $\langle c_1 \rangle \in X_\boxtimes$ and $\langle c_2 \rangle, \dots, \langle c_{k - 2} \rangle \in X_\square$. 
We can scale these vectors so that $\beta(c_1, c_1) = \xi$ and $\beta(c_i, c_i) = 1$ for all $i > 1$. 
Then the map $f$ associated to the matrix $A$ with $c_i$'s as its column is in $G$ and we have $f(\langle x \rangle) = \langle y \rangle$. 
\end{proof}

\begin{lemma}\label{lem:ind}
For any $k \geq 3$, the graph induced on the neighborhood of every vertex of $\Gamma^\square(k, q)$ is isomorphic
  to $\Gamma^\square(k-1, q)$.
\end{lemma}
\begin{proof}
We can choose the vertex as $x = \langle (0, \ldots, 0, 1) \rangle$ since $Q(x) = 1$ is a square and the graph is vertex-transitive  by Lemma \ref{lem:vtx_trans}.
  As $x^\perp = \{ \langle y \rangle : y = (y_1, \dots, y_k) \in \FF_q^k, y_k=0 \}$,
  the quadratic form on $x^\perp$ is $\xi x_1^2 +  \sum_{i=2}^{k-1} x_i^2$, and hence
  the neighborhood of $x$ is isomorphic to $\Gamma^\square(k-1, q)$.
\end{proof}

\begin{lemma}\label{lem:k_eq_2}
  The graph $\Gamma^\square(2, q)$ is $K_2$-free and has $(1+o(1)) q/2$ vertices.
\end{lemma}
\begin{proof}
  Let $\langle (a_1, a_2) \rangle$ be an arbitrary point of $X_\square$.
  The point orthogonal to  $\langle (a_1,a_2) \rangle$
  is $\langle ( a_2, -\xi a_1)\rangle$.
  As $\xi$ is a non-square, $Q(a_2, -\xi a_1)=\xi a_2^2 + \xi^2 a_1^2 =\xi Q(a_1,a_2)$ is a non-square.
  Therefore, $\langle (a_1,a_2) \rangle \in X_\square$ and its orthogonal point $\langle ( a_2, \xi a_1) \rangle \in X_\boxtimes$, which shows that  $\Gamma^\square(2, q)$  has no edges.
Furthermore,  this gives a bijection between $X_\square$ and $X_\boxtimes$.
At most two points $\langle (a_1, a_2) \rangle$ satisfy $0 = Q(a_1, a_2) = \xi a_1^2 +  a_2^2$, and hence $|X_\square \cup X_\boxtimes| \geq q-1$.
Therefore, $|X_\square| = (1+o(1))q/2$.
\end{proof}

By combining Lemmas \ref{lem:ind} and \ref{lem:k_eq_2}, we obtain the following.

\begin{theorem}
  The graph $\Gamma^\square(k, q)$ is $K_k$-free for all $k \geq 2$. \qed
\end{theorem}

\begin{remark}
The graph $\Gamma^\square(1, q)$ has no vertices since for every non-zero square $\lambda \in \mathbb{F}_q$ the element $\xi \lambda$ is a non-square, and hence this graph is $K_1$-free.
We could have started with this as the base case of our induction but we believe that the non-trivial case of $k = 2$ is more instructive.
\end{remark}

To estimate the second largest eigenvalue of our graph we use \textit{interlacing} of eigenvalues.
Let $\Gamma'$ be a graph on $m$ vertices and $\Gamma$ an induced subgraph of $\Gamma'$ on $n$ vertices.
Let $\lambda_1' \geq \lambda_2' \geq \cdots \geq \lambda_m'$ be the eigenvalues of $\Gamma'$, and $\lambda_1 \geq \lambda_2 \geq \cdots \geq \lambda_n$ the eigenvalues of $\Gamma$.
Interlacing says that $\lambda_i' \geq \lambda_i \geq \lambda_{m-n+i}'$, for all $i = 1, \dots, n$ (see for example \cite[Corollary 2.2]{Haemers1995}).

\begin{theorem}
  The graph $\Gamma^\square(k, q)$ is an $(n, d, \lambda)$-graph with $n = \Theta(q^{k - 1})$, $d = \Theta(q^{k - 2})$ and $\lambda = q^{(k - 2)/2}$.
\end{theorem}
\begin{proof}
We know that the number of vertices in $\Gamma = \Gamma^\square(k, q)$ is $(1 + o(1)) q^{k - 1}/2$. 
By Lemma \ref{lem:ind}, the graph is $d$-regular with  $d = (1 + o(1)) q^{k-2}/2$.\footnote{The precise degree of the graph is also given in \cite{Bannai90}.
It corresponds to the value of $k_{(q+1)/2}$ in, depending on the case, Section 4, 5, 6, or 7.
In each section $k_{(q+1)/2}$ is either given in the first lemma or just before the first lemma.
}
Consider the graph $\Gamma'$ with vertex set $X_0 \cup X_\square \cup X_\boxtimes$ and adjacency defined by orthogonality with respect to our quadratic form.
We will first show that the second largest eigenvalue of $\Gamma'$ is  $q^{(k - 2)/2}$ (following the same argument as in \cite[Sec.~2]{Alon1997}) and then use interlacing to deduce that every eigenvalue of $\Gamma$ except $d$ has absolute value at most $q^{(k - 2)/2} = O(\sqrt{d})$.

The graph $\Gamma'$ has $m = q^{k - 1} + \cdots + q + 1$ vertices, and since the points orthogonal to any given point form a hyperplane the graph is $\delta$-regular with $\delta = q^{k -2} + \cdots + q + 1$.
As any two distinct hyperplanes intersect in a codimension $2$ subspace, any two distinct vertices have exactly $\mu = q^{k - 3} + \cdots + q + 1$ common neighbours.
Therefore, the adjacency matrix\footnote{The diagonal entries corresponding to the set of vertices that have a loop around them, that is $X_0$, are equal to $1$.} $A$ of $\Gamma'$ satisfies
\[A^2 = \mu J + (\delta - \mu) I,\]
where $J$ is the all one matrix and $I$ is the identity matrix, of dimension $m \times m$.
The largest eigenvalue of $A$ is $\delta$ as $\Gamma'$ is $\delta$-regular.
Moreover, the graph is clearly connected and thus the eigenspace corresponding to $\delta$ is of dimension $1$.
Let $v$ be an eigenvector with eigenvalue not equal to $\delta$.
Then $Jv = 0$, and hence $A^2v = (\delta - \mu)v$, which implies that the square of the eigenvalue of  $v$ is $\delta - \mu = q^{(k - 2)}$.
Thus, all eigenvalues of $\Gamma'$ except for the largest one have absolute value $q^{(k - 2)/2}$.

Say $\lambda_1' \geq \lambda_2' \geq \cdots \geq \lambda_m'$ are the eigenvalues of $\Gamma'$ and $\lambda_1 \geq \lambda_2 \geq \cdots \geq \lambda_n$ are the eigenvalues of $\Gamma$.
Then by interlacing we have $\lambda_2 \leq \lambda_2' = q^{(k - 2)/2}$ and $-\lambda_n \leq  - \lambda_m' = q^{(k - 2)/2}$.
Therefore, every eigenvalue of $\Gamma$ except $\lambda_1 = d$ has absolute value at most $q^{(k - 2)/2}$. 
\end{proof}

\section{Conclusion}
The first value of $k > 3$ for which we have a separation in the density of our $K_k$-free graphs from Alon's triangle-free graphs is at $k = 5$.
It will be interesting to find a family of $K_4$-free optimally pseudorandom graphs that have a higher density than Alon's graph.
Even more exciting would be to find tight examples for the conjecture, for any $k > 3$, or prove that such examples do not exist.
We believe that graphs coming from finite geometry, especially those related to quadratic forms, can play a role in better constructions.

In a recent work, Mubayi and Versta\"ete \cite{Mubayi-Verstraete} have shown that for any fixed $k \geq 3$, an optimally dense construction of $K_k$-free $(n, d, \lambda)$-graphs would imply the lower bound $R(k, t) = \Omega^*(t^{k - 1})$ on the off-diagonal Ramsey numbers, which matches the best known upper bound of $O^*(t^{k - 1})$.
In fact, any construction with edge density $\Omega(n^{-1/k})$ would already match the best known lower bounds on $R(k, t)$, proved by Bohman and Keevash \cite{Bohman-Keevash}.
Therefore, even such a small improvement on our construction would be very interesting.

\paragraph*{Acknowledgements} We would like to thank David Conlon for his helpful
remarks on an earlier draft of this paper. We would like to thank Akihiro
Munemasa whose work together with the second author in \cite[\S6]{Ihringer2019} on
cospectral graphs inspired the current result.
Finally, we would like to thank the referees for their careful reading and helpful comments.

\bibliographystyle{plain}

\end{document}